\documentclass[a4paper]{article}
\usepackage[margin=1in]{geometry}

\usepackage{authblk}

\usepackage{amsthm} 
    \theoremstyle{plain}
        \newtheorem{theorem}{Theorem}
        \newtheorem{corollary}[theorem]{Corollary}
        \newtheorem{lemma}[theorem]{Lemma}
        \newtheorem{observation}[theorem]{Observation}
        
    \theoremstyle{definition}

    \theoremstyle{remark}
        \newtheorem*{remark}{Remark}
   
    \usepackage{amssymb}

\usepackage{hyperref}
	\hypersetup{
	    colorlinks=true,
	    linkcolor=red,
	    citecolor=red,
	    pdfborder={0 0 0}}
\usepackage[dvipsnames]{xcolor}
\colorlet{R1Color}{black}
\colorlet{R2Color}{black}
\colorlet{OurColor}{black}

\usepackage[normalem]{ulem}
\renewcommand{\sout}[1]{\unskip} %to hide the striked out parts

\usepackage[T1]{fontenc}
\usepackage{graphicx}
\usepackage{todonotes}
\begin{document}
\title{Face-hitting Dominating Sets in Planar Graphs}
%
%\titlerunning{Abbreviated paper title}
% If the paper title is too long for the running head, you can set
% an abbreviated paper title here

\author{P. Francis}

\affil{Department of Mathematics\\
    SAS, VIT-AP University, Amaravati, 
    Andhra Pradesh, India.\\
    francis@vitap.ac.in}

\author{Abraham M. Illickan}
    
\affil{Department of Computer Science\\
    University of California, Irvine.\\
    aillicka@uci.edu}
   \author{Lijo M. Jose}
   \author{Deepak Rajendraprasad}
\affil{  Department of Computer Science\\
    Indian Institute of Technology Palakkad\\
    112004005@smail.iitpkd.ac.in,
    deepak@iitpkd.ac.in}

\maketitle

\begin{abstract}
    A \emph{dominating set} of a graph $G$ is a subset $S$ of its vertices such that each vertex of $G$ not in $S$ has a neighbor in $S$. A \emph{face-hitting set} of a plane graph $G$ is a set $T$ of vertices in $G$ such that every face of $G$ contains at least one vertex of $T$. We show that the vertex-set of every plane (multi-)graph without isolated vertices, self-loops or $2$-faces can be partitioned into two disjoint sets so that both the sets are dominating and face-hitting. We also show that all the three assumptions above are necessary for the conclusion.

    As a corollary, we show that every $n$-vertex simple plane triangulation has a dominating set of size at most $(1 - \alpha)n/2$, where $\alpha n$ is the maximum size of an independent set in the triangulation.     Matheson and Tarjan [European J. Combin., 1996] conjectured that every plane triangulation with a sufficiently large number of vertices $n$ has a dominating set of size at most $n / 4$. Currently, the best known general bound for this is by Christiansen, Rotenberg and Rutschmann \cite{christiansen2024triangulations} [SODA, 2024] who showed that  every plane triangulation on $n > 10$ vertices has a dominating set of size at most $2n/7$. Our corollary improves their bound for $n$-vertex plane triangulations which contain a maximal independent set of size either less than $2n/7$ or more than $3n/7$. 

\end{abstract}

\section{Introduction}
All graphs considered in this paper are finite and undirected. A \emph{dominating set} of a graph $G$ is a subset $S$ of its vertices such that each vertex of $G$ not in $S$ has a neighbor in $S$. The \emph{domination number} $\gamma(G)$ of $G$ is the cardinality of a smallest dominating set in $G$. A $k$-coloring $f:V(G) \to [k]$ of $G$ is called \emph{domatic} if each color-class of $f$ is a dominating set of $G$. A \emph{face-hitting set} of a plane graph $G$ is a set of vertices in $G$ that meets every face of $G$. 
Let $\beta(G)$ denote the size of a smallest face-hitting set in $G$. A $k$-coloring $f:V(G) \to [k]$ is called \emph{polychromatic} if each color-class of $f$ is a face-hitting set of $G$. The polychromatic number $p(G)$ of $G$ is the largest number $k$ such that there is a polychromatic $k$-coloring of $G$.  

Let $f$ be a face in a plane graph $G$. If $G$ is not connected, the boundary of $f$ could be a disjoint union of closed walks. The \emph{length} of $f$ is the total of length of all these walks. A $k$-face (resp. $k^+$-face, $k^-$-face) is a face of length $k$ (resp. at least $k$, at most $k$). The \emph{degree} of a vertex is the number of distinct vertices adjacent to it. The main result of this paper is that every plane graph without isolated vertices, self-loops or $2$-faces has a $2$-coloring which is simultaneously domatic and polychromatic. 

\begin{theorem}                                     \label{thm:General}
    % Every plane graph $G$ with minimum degree at least one, minimum face-length at least three and no self-loops 
    Every plane graph  $G$ without isolated vertices, self-loops or $2$-faces, has two disjoint subsets $V_1, V_2 \subseteq V(G)$, such that both $V_1$ and $V_2$ are dominating and face-hitting.
\end{theorem}

The assumptions made on $G$ in the above theorem are indeed necessary. It follows from Theorem~\ref{thm:General} that every $n$-vertex plane graph $G$ satisfying the premise of the theorem has a subset $S$ of vertices of cardinality at most $n/2$ which is both dominating and face-hitting. It is easy to see that, if we allow isolated vertices or length-$1$ faces, then $\gamma(G)$ can be as large as $n$. If we allow $2$-faces, then $\beta(G)$ can be as large as $3n/4$. For the second observation, consider $G$ to be a disjoint union of $n/4$ components, each component being a $K_4$ in which every edge is replaced with a $2$-face. If we allow self-loops, then there can be a $3^+$-face which has only one vertex in its boundary. For example consider a vertex $v$ with three self loops $l_1, l_2, l_3$ such that $l_1$ and $l_2$ have disjoint interiors, but both are inside of $l_3$. Then the face bounded by $l_1, l_2, l_3$ is such an example. Such a face can be present even if we assume that minimum face-length of the graph is $3$. We can add a neighbor of $v$ inside each of the $1$-faces incident to $v$ so that all the faces have length at least $3$. It is clear that we cannot have a polychromatic $2$-coloring in this case. 

We can also see that the bound of $n/2$ is tight under these assumptions. A disjoint collection of edges (or $3$-length paths, or $4$-length cycles) have $\gamma(G) \geq |V(G)|/2$. Bose et al.~\cite{bose1997guarding} showed that there exists an infinite family of simple connected plane graphs with $\beta(G) \geq \lfloor |V(G)|/2 \rfloor$.

Theorem~\ref{thm:General} has an interesting application. A \emph{near-triangulation} is a simple planar graph embedded in the plane such that all its faces except possibly the outer one are bounded by three edges. A \emph{triangulation} is a near-triangulation in which the outer face is also bounded by three edges. If we remove an independent set from a triangulation $G$ and then apply Theorem~\ref{thm:General} on the resultant graph $G'$, we can show that the dominating and face-hitting sets $V_1,V_2$ of $G'$ will both be dominating sets of $G$. This observation helps us make some progress towards a conjecture by Matheson and Tarjan. Matheson and Tarjan \cite{MathTar} in 1996 proved that every $n$-vertex plane triangulation $G$ has a domatic $3$-coloring (hence $\gamma(G) \leq n/3$) and conjectured that if $n$ is sufficiently large, then $\gamma(G) \leq n / 4$. \v{S}pacapan \cite{vspacapan2020domination} in 2020 showed that  every plane triangulation on $n > 6$ vertices has a dominating set of size at most $17n/53$. Very recently in 2024 Christiansen, Rotenberg and Rutschmann \cite{christiansen2024triangulations} showed that  every plane triangulation on $n > 10$ vertices has a dominating set of size at most $2n/7$. As a consequence of  Theorem~\ref{thm:General}, we improve this bound for $n$-vertex plane triangulations which contain a maximal independent set of size either less than $2n/7$ or more than $3n/7$ and verify Matheson-Tarjan conjecture for $n$-vertex plane triangulations with an independent set of size $n/2$.

The novelty in our result is that we find a single $2$-coloring which is both polychromatic and domatic. In fact, it is easy to show that every plane graph $G$ in this class has a domatic $2$-coloring and a polychromatic $2$-coloring. Any proper $2$-coloring of a spanning forest of $G$ will be a domatic $2$-coloring. Existence of a polychromatic $2$-coloring for $G$ is a solved exercise in Lov{\'a}sz's famous book - \emph{Combinatorial Problems and Exercises} since the first edition \cite{lovasz1979combinatorial}. Domatic colorings and polychromatic colorings on their own have been studied extensively in planar graphs.

The monograph by Haynes et al. \cite{haynes2013fundamentals} is a comprehensive reference on Domination.  Reed  \cite{reed1996paths}, Sohn and Yuan \cite{sohn2009domination}, and Xing et al. \cite{xing2006domination} established upper bounds $3n/8$, $4n/11$ and $5n/14$ respectively on domination number for $n$-vertex graphs with minimum degree at least three, four and five. For triangulated discs, the upper bound of $n/3$ by Matheson and Tarjan is tighter than the above general bounds. The domination number of planar graphs have received special attention. MacGillivray and Seyffarth \cite{macgillivray1996domination} established an upper bound of three and ten respectively on the domination number of planar graphs with diameter two and three. Goddard and Henning \cite{goddard2002domination} improved upon this and showed that there is only one planar graph of diameter two with domination number three and also showed that every sufficiently large planar graph of diameter three has domination number at most seven. The dominating set problem on planar graphs is NP-complete. Fomin and Thilikos \cite{fomin2006dominating} showed that the $k$-dominating set problem on planar graphs can be solved in time $O(2^{15.13\sqrt{k}} + n^3)$. Variants of dominating sets like independent dominating sets and total dominating sets in planar triangulations have also been studied. Botler, Fernandes and Guti{\'e}rrez \cite{botler2023independent} proved that every planar triangulation $G$ on $n$ vertices have an independent dominating set of size less than $3n/8$ and this could be improved to $n/3$ if the minimum degree of $G$ is at least five. Claverol et al.~\cite{CLAVEROL2021112179} showed that any near-triangulation of order $n$ has a total dominating set of size at most $2n/5$, with two exceptions. Francis et al.~\cite{francisJGT} showed that every near-triangulation with minimum degree at least three, has two disjoint total dominating sets.

Face hitting sets have an interesting connection with terrain guarding problems. Bose et al. \cite{bose1997guarding} used the idea of polychromatic colorings of triangulations in the context of guarding polyhedral terrains. They showed that there is an infinite family of plane triangulations where the smallest face hitting set has size $\lfloor n/2 \rfloor$.  For a plane graph $G$, let $g(G)$ denote the size of the smallest face in $G$. Alon et al. \cite{noga2009ploychromatic} proved that for any plane graph $G$ with $g(G)$ at least $3$,  $p(G) \geq \left\lfloor\frac{3g-5}{4}\right\rfloor$. They also showed that this bound is nearly tight, as there are plane graphs for which $p(G) \leq \left\lfloor\frac{3g+1}{4}\right\rfloor$. It is easy to observe that for every plane triangulation $G$, $p(G)$ is either $2$ or $3$. As a consequence of Heawood's theorem \cite{heawood1898four}, $p(G)=3$ for a plane triangulation $G$ if and only if $G$ is Eulerian (degree of each vertex is even)  \cite{noga2009ploychromatic}. Hoffmann and Kriegel \cite{hoffmann1996graph} showed that every $2$-connected bipartite plane graph admits a polychromatic $3$-coloring by showing that every such graph can be transformed into an Eulerian triangulation by adding edges only. The decision problem whether a plane graph is polychromatic $k$-colorable is NP-complete when $k \in \{3,4\}$ \cite{noga2009ploychromatic}. Horev and Krakovski \cite{horev2009polychromatic} proved that every plane graph of degree at most $3$, other than $K_4$ and a subdivision of $K_4$ on five vertices, admits a polychromatic $3$-coloring. Horev et al. \cite{horev2012polychromatic} proved that every $2$-connected cubic bipartite plane graph admits a polychromatic $4$-coloring. This result is tight, since any such graph must contain a face of size $4$.

% Every planar graph on $n$ vertices has an independent set of size at least $n/4$ which is a consequence of the Four Color Theorem \cite{AppHakKoc}. Interestingly, the maximum independent sets have received considerably more attention in triangle-free planar graphs than in triangulations. From Gr{\"o}tzsch's theorem that every triangle-free planar graph is $3$-colorable \cite{grotzsch1959dreifarbensatz}, it follows that $n$-vertex triangle-free planar graphs have an independent set of size at least $n/3$. Heckman and Thomas \cite{heckman2006independent} proved that every triangle-free planar graph on $n$ vertices with maximum degree three has an independent set with size at least $3n/8$. Steinberg and Tovey \cite{steinberg1993planar} proved that every triangle-free planar graph has an independent set of size at least $(n+1)/3$ and showed that this lower bound is tight for an infinite families of graphs $\mathcal{G}$. Dvo{\v r}{\'a}k et al., \cite{dvovrak2019triangle} improved the same to  $(n+2)/3$ except for the same infinite families of graphs $\mathcal{G}$. Finding the size of a maximum independent set is NP-complete in planar graphs and same is true even if we restrict to the triangle-free or cubic planar graphs \cite{madhavan1984approximation}.
% Dvo{\v r}{\'a}k  and Mnich \cite{dvorak2017large} proved that if a triangle-free planar graph of order $n$ does not have an independent set larger than $(n + k)/ 3$, then its tree-width is $O (\sqrt k )$. 
 
\subsection{Terminology and notation}
Let $G$ be a graph. The vertex-set and the edge-set of $G$ are denoted respectively by $V(G)$ and $E(G)$. The open neighborhood (resp. closed neighborhood) of a vertex $v$ in graph $G$ is denoted by $N_G(v)$ (resp. $N_G[v]$). The \emph{degree} $d_G(v)$ of a vertex $v$ in $G$ is $|N_G(v)|$. A set of vertices $I$ of a graph $G$ is called an \emph{independent set} if no two vertices in $I$ are adjacent. A graph is \emph{planar} if it can be embedded on the plane in such a way that no two edges cross. A plane graph $G$ is a planar graph together with such an embedding.  We denote the set of vertices lying on the boundary of a face $f$ as $V(f)$. A cycle of length exactly $k$, at least $k$ and at most $k$ in $G$ are respectively termed \emph{$k$-cycle}, \emph{$k^+$-cycle} and \emph{$k^-$-cycle}.

\section{Proof of Theorem~\ref{thm:General}}
In this section, we prove the following theorem, which is a restatement of Theorem~\ref{thm:General} in the language of vertex coloring.

\begin{theorem}                                     \label{thm:coloring}
 % If $G$ is a plane graph with minimum degree at least one, minimum face-length at least three and without self-loops,   
 If $G$ is a plane graph without isolated vertices, self-loops or $2$-faces, then $G$ has a $2$-coloring which is simultaneously domatic and polychromatic.  
\end{theorem}

The main part of the paper is devoted to proving the next lemma from which Theorem~\ref{thm:coloring} will follow easily. We call a $2$-coloring of a plane graph \emph{$3^+$-polychromatic} if every $3^+$-face  contains vertices of both colors.

\begin{lemma}                                     \label{lem:coloring}
If $G$ is a connected plane graph without self-loops and with minimum degree at least two, then $G$ has a $2$-coloring which is simultaneously domatic and $3^+$-polychromatic. 
%If $G$ is a plane graph with minimum degree at least two and minimum face-length at least three, then $G$ has a $2$-coloring which is simultaneously domatic and polychromatic. 
\end{lemma}

To prove Lemma~\ref{lem:coloring} we construct a supergraph $G'$ from $G$ by adding more edges to $G$ without violating planarity. An edge $uw$ is only added to $G$ if it forms a facial triangle $uvw$ where $\{uv,vw\}\subseteq E(G)$.  We call the edges in $E(G)$ as \emph{true edges} and those in $E(G') \setminus E(G)$ as \emph{dummy edges}.  Vertices $v$ and $u$ are called \emph{true neighbors} in $G'$ if $vu$ is a true edge. We call a vertex $v$ \emph{happy} if there exists a triangle $uvw$ in $G'$ where $uv$ and $vw$ are true edges. A face $f$ in $G$ with vertices $V(f)$ is \emph{happy} if the subgraph of $G'$ induced on $V(f)$ contains a triangle. A \emph{true angle at $v$} in $G'$ is an angle between two cyclically consecutive true edges incident on $v$ which connects $v$ to two distinct vertices (ignoring any dummy edges between them). Let $\mathcal{G}$ be the family of all plane multigraphs with maximum number of happy vertices and happy faces that can be obtained by adding dummy edges to $G$. Let $G'$ be a graph in $\mathcal{G}$ with the smallest numbers of dummy edges. We make the following observations about $G'$.

\begin{observation}
\label{obs:0}
Every $3^+$-face of $G$ is happy in $G'$.
\end{observation}
\begin{proof}
Suppose $f$ is a $3^+$-face of $G$ which is unhappy in $G'$. Since $G$ is connected, the boundary $B$ of $f$ is a single closed walk. Further since $G$ does not have self loops and $f$ has length at least three, $V(f)$ has at least three vertices. Hence we can make $f$ happy by adding a dummy edge inside $f$ between two distinct vertices in $V(f)$ which are at distance two along the walk $B$. The resulting graph is planar supergraph of $G$ and has more happy faces than $G'$. This violates the membership of $G'$ in $\mathcal G$.
\end{proof}

\begin{remark}
Notice that both the assumptions on $G$ are necessary for Observation~\ref{obs:0}. Recall that if we allow self-loops, we can have $3^+$-faces bounded by a single vertex, and such faces cannot be made happy. If we allow disconnected graphs, then a $4$-face bounded between two $2$-cycles cannot be made happy by adding a single dummy edge. 
\end{remark}

Unlike $3^+$-faces, we cannot guarantee that all the vertices are happy in $G'$. The next observation, even though technical, illustrates precisely what happens in the neighborhood of an unhappy vertex.

\begin{observation}                                 \label{obs:1}
  If $v$ is an unhappy vertex in $G'$, then 
  \begin{enumerate}
  \item there is exactly one dummy edge incident on $v$ through every true angle at $v$, and
  \item each of these dummy edges makes exactly one true neighbor
  of $v$ happy.
  \end{enumerate}
\end{observation}

\begin{proof}
     Let $D_v$ denote the set of dummy edges incident to $v$ in $G'$ and $H_v$ denote the set of happy true neighbors of $v$. If there are no dummy edges incident on $v$ through a true angle $uvw$, we can add a dummy edge $uw$ without violating planarity, making $v$ happy. This will increase the number of happy vertices contradicting the membership of $G'$ in $\mathcal G$. Since the number of true angles at $v$ is at least $|N_G(v)| = d_G(v)$, we have $|D_v| \geq d_G(v)$. Further, since $|H_v| \leq d_G(v)$, we have $|D_v| \geq |H_v|$.

     Every vertex in $H_v$ needs at most one edge from $D_v$ to become happy. No vertex outside $H_v$ can be made happy by an edge in $D_v$. Hence, either if $|D_v| > |H_v|$ or if $|D_v| = |H_v|$ and one of the edges in $D_v$ is making two true neighbors of $v$ happy, then at least one edge in $D_v$ is redundant. That is, we can remove this edge and still leave all the vertices in $H_v$ happy. If this deletion does not leave a true angle at $v$ without a dummy edge, then the happiness of the corresponding face is also intact. This contradicts the choice of $G'$ as a smallest member in $\mathcal G$. On the other hand, if this deletion leaves a true angle at $v$ without a dummy edge, we can make $v$ happy as we did in the first case and restore the happiness of the corresponding face. This contradicts the membership of $G'$ in $\mathcal G$. Hence $|D_v| = |H_v|$ and each edge in $D_v$ makes exactly one vertex in $H_v$ happy. Since $d_G(v)$ is sandwiched between $|H_v|$ and $|D_v|$, we also have $|H_v| =|D_v| = d_G(v)$.
\end{proof}

The next observation is an easy restatement of the equality $|H_v| = d_G(v)$ that we established in Observation~\ref{obs:1}
\begin{observation}\label{obs:2}
    Two unhappy vertices cannot be true neighbors in $G'$.
\end{observation}

% \begin{observation}\label{obs:3}
%      No dummy edges are added between two separate connected components of $G$. 
%  \end{observation}
%  \begin{proof}
%      A dummy edge is added only between two vertices which share a true common neighbor, hence they cannot be from two separate components in $G$.
%  \end{proof}
 \begin{observation}[Key Observation]\label{obs:4}
 % Each connected component of $G'$ has at most one unhappy vertex.
     $G'$ has at most one unhappy vertex.
 \end{observation}

 \begin{proof}
Let $v$ be an unhappy vertex. From Observation~\ref{obs:1}, it is clear that if we remove one dummy edge $e$ incident on $v$, the only effect on happiness is that one happy true neighbor of $v$, say $u$, becomes unhappy and the face $f$ in $G$ containing $e$ may become unhappy. Moreover, one true angle at $v$, say $uvw$ becomes free. We can now add the dummy edge $uw$ inside $f$ to make $v$ and $f$ happy. So, the unhappiness of a vertex $v$ can be shifted to any one of its true neighbors without creating any other unhappy vertices or faces. Note that this shifting does not increase the total number of dummy edges. 

Suppose there were two unhappy vertices, say $v, v'$ in $G'$. Since $G$ is connected, there is a path $P = (v = x_0, \ldots, x_k = v')$ in $G$. Then we can do the above shifting of unhappiness repeatedly from $x_i$ to $x_{i+1}$ for $0 \leq i \leq k-2$. This shifting will end when $x_{k-1}$ and $v'$ are unhappy. This new graph also qualifies to be $G'$ and it contradicts Observation~\ref{obs:2}. Hence, there can be at most one unhappy vertex in $G'$.  
\end{proof}

By the four color theorem \cite{AppHak,AppHakKoc}, there exists a proper $4$-coloring $\phi : V[G'] \to \{1,2,3,4\}$ of $G'$. By Observation~\ref{obs:4}, we can assume without loss of generality that %each
the unhappy vertex in $G'$ gets color $1$ and has a true neighbor of color $2$. Obtain a $2$-coloring $\psi$ of $G'$ by merging the color classes $1$ and $3$ to a single color class and $2$ and $4$ to a different color class. This ensures that, in $\psi$, the unhappy vertex sees both colors in its closed neighborhood $N_G[v]$. Every happy vertex $v$ is part of a triangle $uvw$ in $G'$ with $uv, vw$ being true edges. Since $u, v, w$ get three different colors in $\phi$, $N_G[v] \supseteq \{u, v, w\}$ will contain vertices of both colors in $\psi$.
By Observation~\ref{obs:0}, every $3^+$-face in $G$ got at least one triangle among its boundary vertices in $G'$. Hence every $3^+$-face of $G$ also sees both the colors in $\psi$. Hence $\psi$ is a domatic as well as a $3^+$-polychromatic $2$-coloring of $G$. This completes the proof of Lemma~\ref{lem:coloring}.

To prove Theorem~\ref{thm:coloring}, we need to extend Lemma~\ref{lem:coloring} to handle graphs with multiple components and  pendant vertices. Let $G$ be a plane graph without isolated vertices, self-loops or $2$-faces. Trim the graph by recursively removing vertices of degree one to obtain a subgraph $G'$ with no degree one vertex in it. Since we are only degree one vertices at every stage, we do not create any new components. $G'$ can have only two types of components - components with minimum degree at least two or isolated vertices (trivial components). 

Consider a nontrivial component $H$ of $G'$. Given a (partial) $2$-coloring of the vertices of $G$, a vertex $v$ (resp. face $f$) is said to be \emph{satisfied}, if $N_G[v]$ (resp. $V(f)$) contains at least one vertex of each color. By Lemma~\ref{lem:coloring} we know that there exists a $2$-coloring of $H$ satisfying every $3^+$-face and vertex of $H$. Notice that $H$ may have $2$-faces even though $G$ does not.  

We start with such a coloring for every nontrivial component in $G'$ and color the isolated vertices in $G'$ with any of the two colors. We then add back the deleted vertices one at a time in the reverse order as we deleted them. Every time we add back a vertex, we color it with the color different from that of its parent. This makes sure that every vertex added back and its parent are satisfied. Since there are no isolated vertices in $G$, every vertex of $G$ is satisfied in the resulting coloring of $G$. 

Now consider the faces. Since we are adding back degree one vertices to $G'$ recursively to obtain $G$, we are not creating any new faces even though some of the faces may have a larger boundary due to the new edges added inside them.  Note that the vertices which were on the boundary of any face in $G'$ still remains in its boundary in $G$. So every face which was satisfied in $G'$ is satisfied in this coloring of $G$. Every face in $G'$ whose boundary contains a closed walk of length at least three is satisfied in $G'$ since that closed walk is the boundary of a $3^+$-face in one of the nontrivial components of $G'$. Let $f$ be an unsatisfied face in $G'$. By the previous observation, $f$ is bounded by a collection consisting of only $2$-cycles and isolated vertices. Since $G$ does not have any $2^-$-faces, $f$ will either contain a pendant vertex of $G$ inside it or $V(f)$ contains vertices from more than one component of $G$. In the first case, we colored this pendant vertex differently from its parent and hence $f$ will be satisfied. In the second case, we can flip the colors of one of the interior components and all the components nested inside it, if needed, to ensure that $f$ is satisfied. 
% Note that since we are flipping the colors of all components nested inside we will not create any new unsatisfied face after the flipping. 
This completes the proof of Theorem~\ref{thm:coloring}.

% the $3^-$-faces. and the faces which only saw isolated vertices, self loops or $2$-cycles in its boundary. Consider any $3^-$-face $f'$ in $G'$. Since, there were only $3^+$-faces in $G$ and no new faces were created in $G'$, there exist a corresponding $3^+$-face $f$ in $G$. Since $f$ is a $3^+$-face, at least one vertex $v$ was deleted from its boundary to obtain $f'$. While adding back $v$ we colored it with a different from its parent $u$. $u$ and $v$ are on the boundary of $f$ and hence it is satisfied. If   Hence every face of $G$ also gets satisfied in this coloring. 

\section{Application to Matheson-Tarjan Conjecture}

As a corollary of Theorem~\ref{thm:General} we can give a conditional upper bound on the domination number of plane triangulations.

\begin{corollary}
\label{cor:MathTarj}
Every $n$-vertex (simple) plane triangulation $G$ with an independent set of size at least $\alpha n$ has $\gamma(G) \leq (1 - \alpha)n/2$.
\end{corollary}
 \begin{proof}
    Consider a plane triangulation $G$ on $n$ vertices with an independent set $I$ of size at least $\alpha n$. We delete $I$ from $G$ to get a plane graph $G'$. Note that since $G$ is a triangulation and every vertex is part of a triangle, removing $I$ from $G$ will not create any isolated vertices in $G'$. Moreover, $G$ has no $2$-faces and since $G$ is a simple triangulation, every vertex removed from $G$ has degree more than two. Hence $G'$ has no $2$-faces. Corresponding to every deleted vertex $v$, a unique face $f_v$ is formed in $G'$ where $V(f_v)=N_G(v)$. Note that $|V(G')| \leq (1-\alpha)n$. By Theorem~\ref{thm:General} we know that there is a face-hitting dominating set $S$ in $G'$ with $|S| \leq (1-\alpha)n/2$. 
    Since $S$ is a dominating set in $G'$ and $G'$ is a subgraph of $G$, $S$ dominates all the vertices of $V(G') \setminus S = V(G) \setminus (I \cup S)$ in $G$ as well. Since $S$ is a face-hitting set in $G'$, it dominates every vertex of $I$ in $G$. Hence $S$ is a dominating set of $G$.
 \end{proof} 

Corollary~\ref{cor:MathTarj} proves Matheson-Tarjan Conjecture for plane triangulations which have an independent set of size at least $n/2$. Recall that Christiansen et al. \cite{christiansen2024triangulations} showed that  every plane triangulation on $n > 10$ vertices has a dominating set of size at most $2n/7$. Corollary~\ref{cor:MathTarj} improves this bound when $\alpha > 3/7$. This includes triangulations obtained by adding a new vertex inside every face of a planar graph whose average face length is slightly below $14n/3$
(and connecting it to the vertices on this face). Since a maximal independent set in a graph is a dominating set, our bound improves Christiansen et al.'s bound when $G$ contains a maximal independent set of size either less than $2n/7$ or more than $3n/7$. 

% ---- Bibliography ----
%
% BibTeX users should specify bibliography style 'splncs04'.
% References will then be sorted and formatted in the correct style.
%
\bibliographystyle{splncs04}
\bibliography{bibtex}
\end{document}